\documentclass{article}

\usepackage{amssymb} 
\usepackage{amsmath} 
\usepackage{latexsym} 
\usepackage{theorem} 
\usepackage{color}
\usepackage{epsfig}

\theorembodyfont{\itshape} 

\newtheorem{theorem}{Theorem}[section]
\newtheorem{lemma}[theorem]{Lemma}

\newtheorem{definition}[theorem]{Definition}

\newtheorem{remark}[theorem]{Remark}
\newenvironment{proof}{{\par\addvspace{0.1cm}\noindent \bf Proof. }}{\hfill$\Box$\par\medskip} 
 
\def\n{m}

\def\an{{\alpha-\n}}
\def\a{\alpha}

\def\Om{\Omega}
\def\RR{\mathbb{R}}

\def\interior#1{\stackrel{\circ}{#1}}

\def\d{\delta}

\title{\bf Minimal unfolded regions of a convex hull and parallel bodies}
\author{Jun O'Hara}

\numberwithin{equation}{section}

\begin{document}

\maketitle

\begin{abstract} 
The {\em minimal unfolded region} (or the {\em heart}) of a bounded subset $\Om$ in the Euclidean space is a subset of the convex hull of  $\Om$ the definition of which is based on reflections in hyperplanes. 
It was introduced to restrict the location of the points that give extreme values of certain functions, such as potentials whose kernels are monotone functions of the distance, and solutions of differential equations to which Aleksandrov's reflection principle can be applied. 
We show that the minimal unfolded regions of the convex hull and parallel bodies of $\Om$ are both included in that of $\Om$. 
\end{abstract}




\section{Introduction}
Let $\Om$ be a compact subset of $\RR^\n$ $(\n\ge2)$. 
The minimal unfolded region of $\Om$ was defined independently by the author for general case (\cite{O1}), and by Brasco, Magnanini, and Salani for the convex case (\cite{BMS}), where it is called the {\em heart} of $\Om$ and denoted by $\heartsuit \Om$. 
The definition is based on the reflections in hyperplanes. Roughly speaking, it is a complement of a region that has no chance to have points where some functions, which can be defined somehow symmetrically, take the extreme values.

\smallskip

We use the following notation. 
Let $S^{\n-1}$ be the unit sphere in $\RR^\n$. 
For a subset $X$ of $\RR^\n$, let $\interior{X},\overline{X}$, and $\textrm{conv}(X)$ be the interior, the closure, the convex hull of $X$, respectively. 
For a unit vector $v\in S^{\n-1}$ and a real number $b$, put $\textrm{cap}^+_{v,b}(X)=X\cap\{x\in\RR^\n\,:\,x\cdot v> b\}$. 
We write $\widetilde X=\textrm{conv}(X)$ and $X^+_{v,b}=\textrm{cap}^+_{v,b}(X)$ when there is no fear of confusion. 
Let $\textrm{R}_{v,b}$ be a reflection of $\RR^\n$ in a hyperplane $\{x\in\RR^\n\,:\,x\cdot v=b\}$. 

\begin{definition}\label{def_folded_core} \rm (\cite{O1},\cite{BMS})  
Let $X$ be a bounded subset of $\RR^\n$ $(\n\ge2)$ and $v\in S^{\n-1}$. 
Define the {\em level of the maximal cap in direction $v$} by 
$$
l_{X}(v)=\inf\{a\,:\,\textrm{R}_{v,b}(X_{v,b}^+)\subset X\>\>\mbox{for every $b\ge a$}\}.
$$
Define the {\em minimal unfolded region} of $X$ by 
$$\textsl{Uf}\,(X)=\bigcap_{v\in S^{\n-1}}\left\{x\in\RR^\n\,:\,x\cdot v\le l_{X}(v)\right\}.$$
\end{definition}

\begin{figure}[htbp]
\begin{center}
\begin{minipage}{.45\linewidth}
\begin{center}
\includegraphics[width=.7\linewidth]{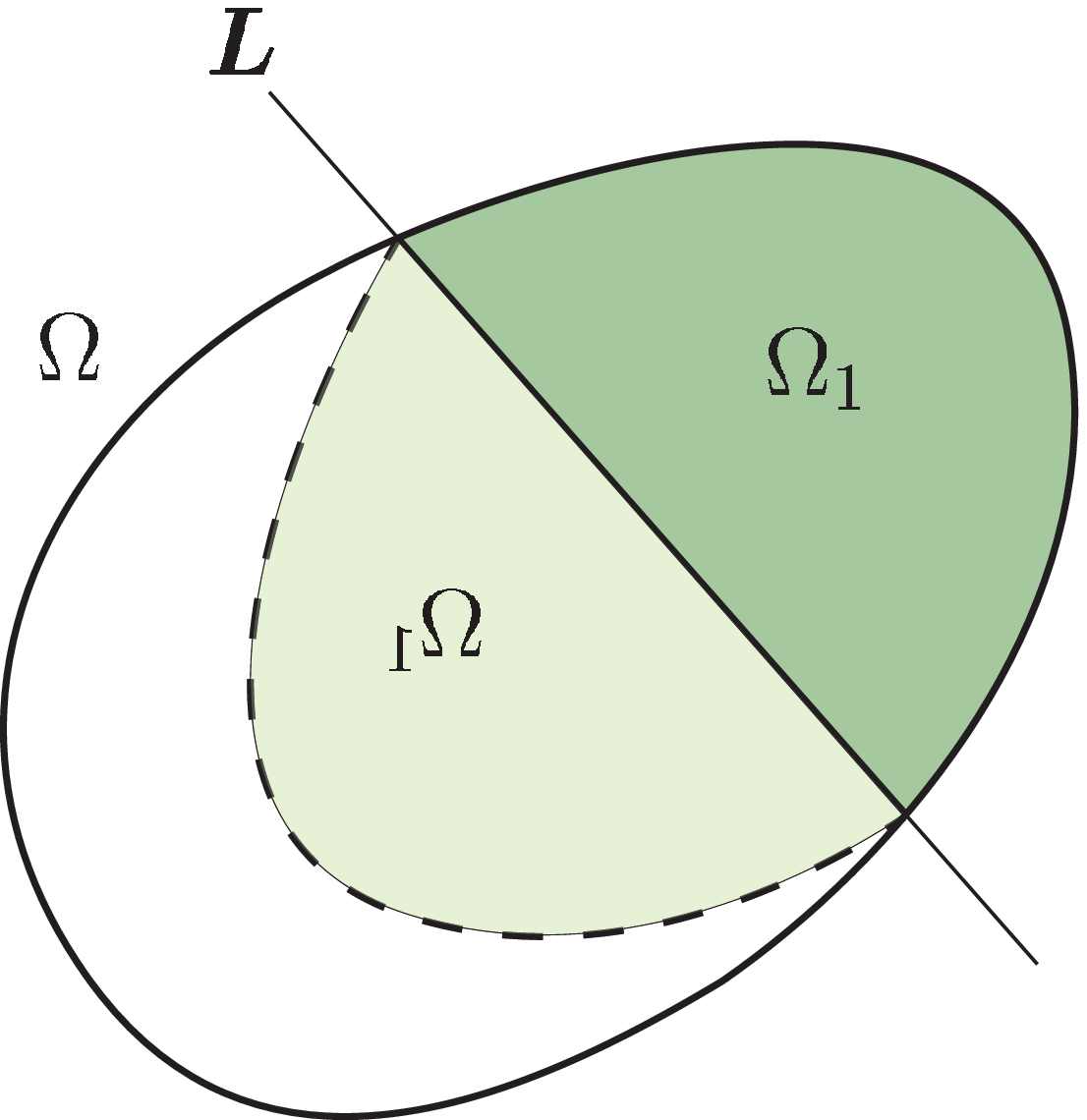}
\caption{Folding a convex set like origami}
\label{origami_convex}
\end{center}
\end{minipage}
\hskip 0.4cm
\begin{minipage}{.45\linewidth}
\begin{center}
\includegraphics[width=\linewidth]{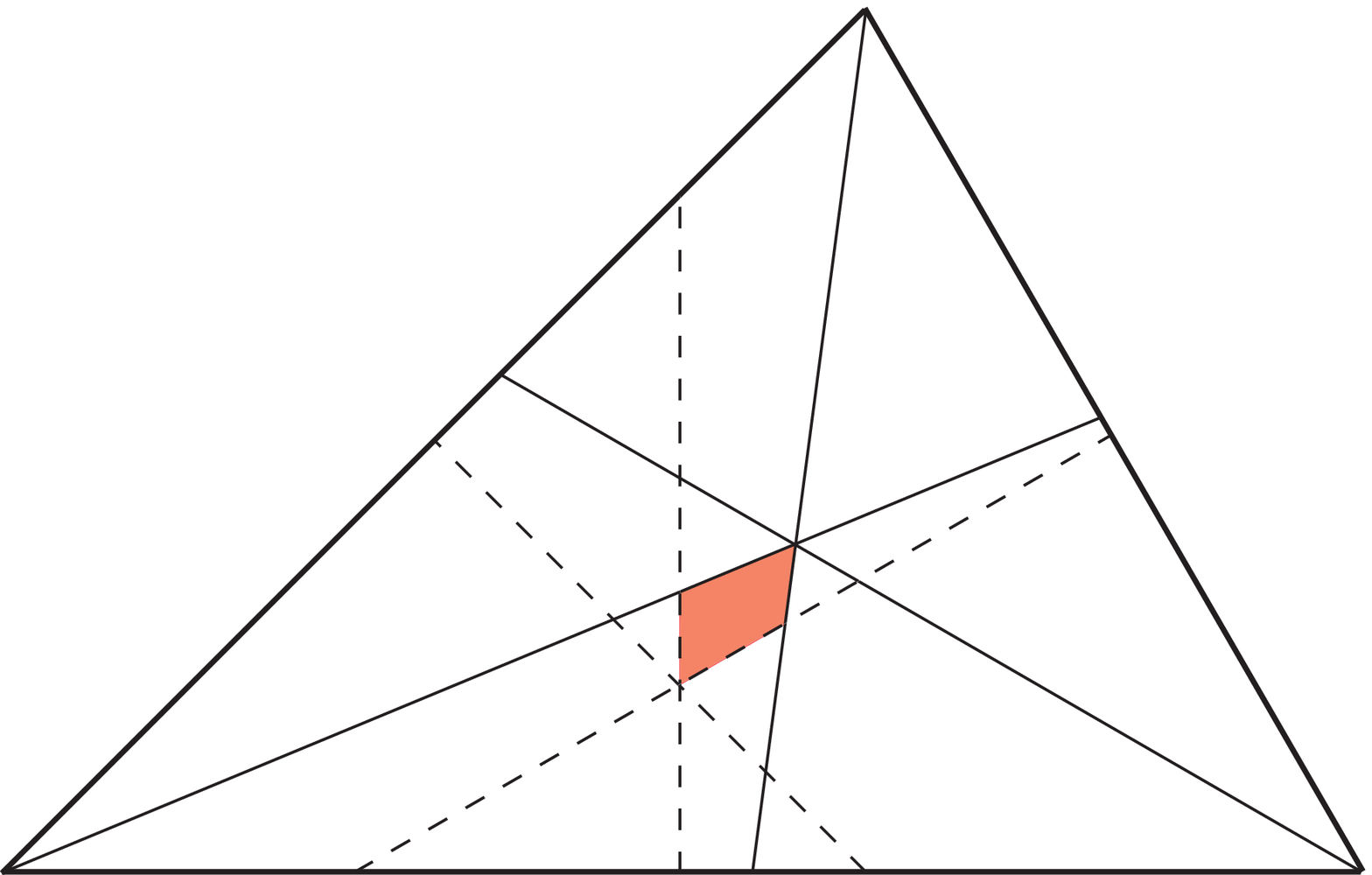}
\caption{A minimal unfolded region of an acute triangle is a quadrilateral formed by two angle bisectors (bold lines) and two perpendicular bisectors (dotted lines). }
\label{folded_core_triangle}
\end{center}
\end{minipage}
\end{center}
\end{figure}

The set $\Om_{v,{l_{\Om}(v)}}^+$ is called the {\em maximal cap} in direction $v$. 
The level $l_\Om(v)$ plays an important role in this article. 
When $\Om$ is compact, it satisfies 
\[
l_\Om(v)=\min\left\{a\,:\,\textrm{R}_{v,b}\Big(\,\overline{\Om_{v,b}^+}\,\Big)\subset\Om\>\>\mbox{for every $b\ge a$}\right\}.
\]

By definition, $\textsl{Uf}\,(\Om)$ is compact and convex. 
It is not empty since it contains the center of mass of $\Om$. 
It is not necessarily contained in $\Om$ if $\Om$ is not convex, as is illustrated in Figure \ref{folded_core_non_convex}. 
However, it is included in the convex hull $\widetilde\Om$ of $\Om$. The reason is as follows. 
Let $M_\Om(v)=\textrm{max}_{x\in\Omega}x\cdot v$ be the support function of $\widetilde\Om$. 
Then $\widetilde\Om=\bigcap_{\,v\in S^{\n-1}}\{x\,:\,x\cdot v\le  M_\Om(v)\}$. Since $l_{\Om}(v)\le M_\Om(v)$ we have $\textsl{Uf}\,(\Om)\subset \widetilde\Om$. 
We remark that $Uf(\Om)$ may not be contained in the interior of $\Om$. For example, the minimal unfolded region of an obtuse triangle has a non-empty intersection with the longest edge, as was mentioned in \cite{BM}. 

If $\Om$ is convex and symmetric in a $q$-dimensional hyperplane $H$ ($q<\n$) then $\textsl{Uf}\,(\Om)$ is included in $H$. 
Especially, the minimal unfolded region of an $\n$-ball consists of the center. 
The convexity assumption cannot be dropped, as is illustrated in Figure \ref{folded_core_non_convex}. Some other properties when $\Om$ is convex can be found in \cite{BM}.

\smallskip
Let us give some examples of functions such that the extreme values are attained only in the minimal unfolded region of $\Om$. 
This follows from the symmetry of the functions. 
%

\smallskip
(1) The potentials whose kernels are monotone functions of the distance. They include, for example, the log potential $\RR^\n\ni x\mapsto\int_\Om\log|x-y|\,dy$, where $dy$ is the standard Lebesgue measure of $\RR^\n$, Riesz-type potential $\int_\Om{|x-y|}^{\an}\,dy$ $(\a>0)$, its generalization to the case when $\a\le0$ and $x\in{\interior\Om}$ 
that can be obtained after renormalization (\cite{O1}), and Poisson integral $\int_\Om t(|x-y|^2+t^2)^{-(\n+1)/2}\,dy$ $(t>0)$. 
We remark that the center of mass is the unique minimum point of $\int_\Om{|x-y|}^{2}\,dy$. 
The existence of maximum points of the renormalization of $\int_\Om{|x-y|}^{\an}\,dy$ $(\alpha\le0)$ restricted to $\interior\Om$ (\cite{O1}) shows $Uf(\Om)\cap \,\interior\Om\ne\emptyset$.  
The properties of the points where the extreme values are attained have been studied by Brasco and Magnanini (\cite{BM}), the author (\cite{O1,O2}), and Sakata (\cite{S,S2}). 
To be precise, the author studied the uniqueness of the extremal point of the Riesz-type potential and its generalization on a general set $\Om$ (with some regularity condition on the boundary), Brasco and Magnanini studied potentials with a monotone kernel when $\Om$ is a convex body, and Sakata gave some sufficient conditions for the uniqueness of the extremal point for reasonably wide class of potentials.

\smallskip
(2) The solutions of differential equations such that one can apply Aleksandrov's reflection principle (or the moving plane method). Here we assume that $\Om$ is a bounded domain in $\RR^\n$ to be compatible with literatures in differential equations. 

The solution $u(x,t)$ of the following heat equation: \setlength\arraycolsep{1pt}
\[
\left\{
\begin{array}{rll}
u_t&=\Delta u &\hspace{0.4cm}\mbox{in $\>\Om\times(0,\infty)$},\\
u&=1 &\hspace{0.4cm}\mbox{on $\>\Om\times\{0\}$},\\
u&=0 &\hspace{0.4cm}\mbox{on $\>\partial\Om\times(0,\infty)$}.\\
\end{array}
\right.
\]
It is the temperature of a heat conductor $\Om$ at time $t$ under the assumption that it is equal to $1$ at time $0$ on $\Om$ and is constantly equal to $0$ on the boundary. 
A point where $u(\,\cdot\,,t)$ attains the maximum is called a {\em hot spot}. 
When $\Om$ is convex, it is known that there is a unique hot spot. 
Brasco, Magnanini, and Salani showed that the hot spot is contained in $Uf(\Om)$ when $\Om$ is convex using Aleksandrov's reflection principle (\cite{BMS}). 

They showed that one can also apply Aleksandrov's reflection principle to the solutions of 
\begin{equation}\label{pde}
\begin{array}{c}
\Delta u+f(u)=0 \>\>\mbox{and $u>0$ in $\Om$},\\
u=0 \>\>\mbox{on $\partial\Om$},
\end{array}
\end{equation}
where $f$ is a locally Lipschitz continuous function, which implies that the maximum points of solutions of \eqref{pde} belong to $Uf(\Om)$. 
A remarkable case is that of the first Dirichlet eigenfunction of the Laplacian, $f(t)=\lambda_1(\Om)t$. \\ \indent
In both cases, since their argument is based on the moving plane method, the same assertion also holds even if $\Om$ is not convex. Namely, hot spots or the maximum points of solutions of \eqref{pde} are contained in $Uf(\Om)\cap \Om$.

\bigskip
Let us assume that $\Om$ is a compact subset of $\RR^\n$ again. 
In this article we give two types of operations that make the domain bigger and the minimal unfolded region smaller. First, the minimal unfolded region of the convex hull of $\Om$ is included in that of $\Om$, as is illustrated in Figures \ref{folded_core_non_convex} and \ref{folded_core_conv_hull}, which was an open problem in \cite{O1}. 
Second, if $\Om_\d$ denote the closure of a $\d$-tubular neighbouhood of $\Om$, 
\begin{equation}\label{eq_def_parallel_body}
\Om_\d
=\{x+\d u\,:\,x\in \Om, u\in B^\n\}=\bigcup_{x\in \Om}B_\d(x),
\end{equation}
where $B^\n$ and $B_\d(x)$ are the unit $\n$-ball and an $\n$-ball with center $x$ and radius $\d$, respectively, 
then the minimal unfolded regions of $\Om_\d$ are included in that of $\Om$ for any $\d>0$. In particular, when $\Om$ is convex, the minimal unfolded region of $\Om_\d$ is equal to that of $\Om$. In this case, $\Om_\d$ is called a {\em $\d$-parallel body} of $\Om_\d$, and is an important notion in convex geometry. 
\begin{figure}[htbp]
\begin{center}
\begin{minipage}{.45\linewidth}
\begin{center}
\includegraphics[width=.7\linewidth]{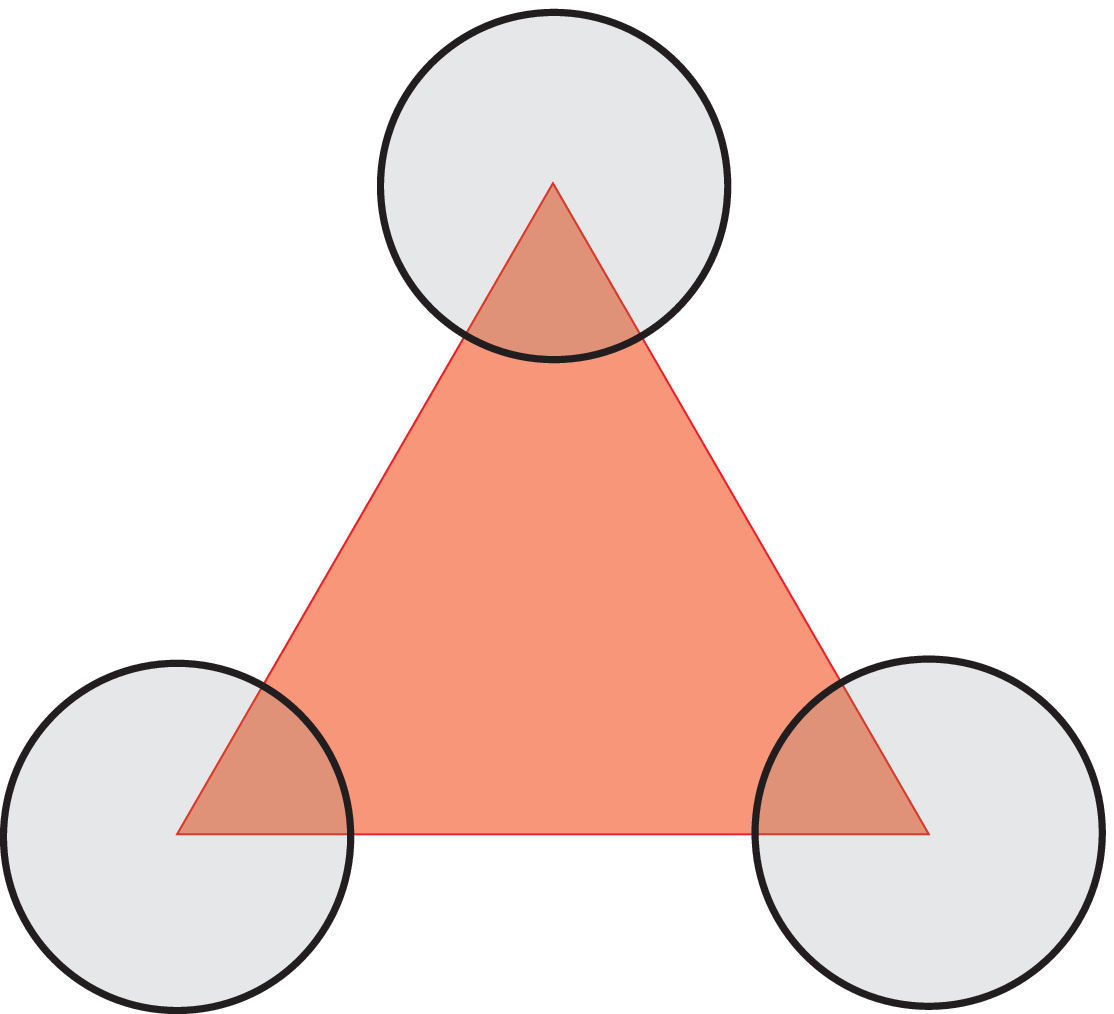}
\caption{The minimal unfolded region (inner regular triangle) of a non-convex set (union of three congruent discs whose centers are located on the vertices of the regular triangle)}
\label{folded_core_non_convex}
\end{center}
\end{minipage}
\hskip 0.4cm
\begin{minipage}{.45\linewidth}
\begin{center}
\includegraphics[width=.7\linewidth]{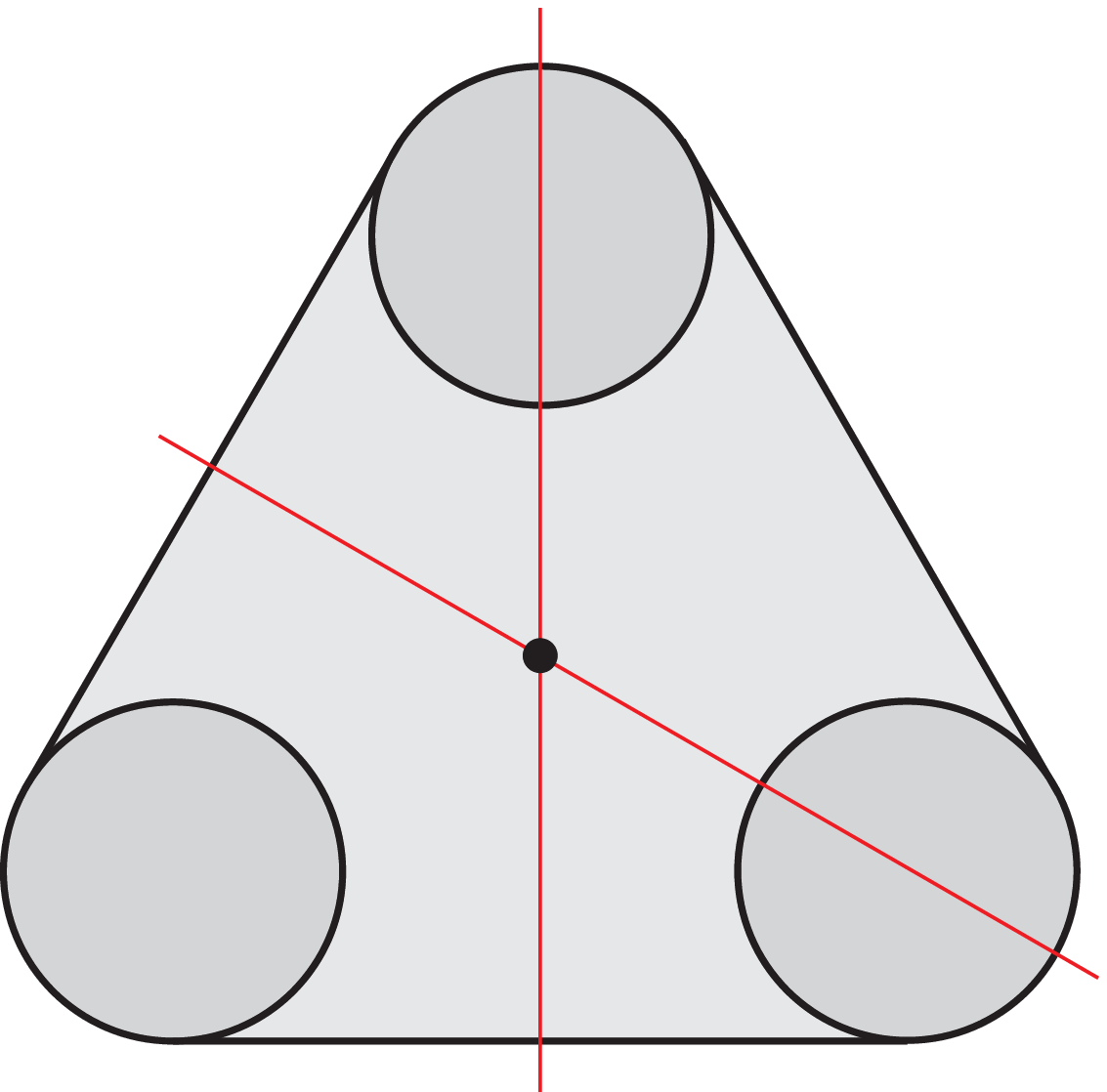}
\caption{The minimal unfolded region of the convex hull consists of a point which is the intersection of lines of symmetry}
\label{folded_core_conv_hull}
\end{center}
\end{minipage}
\end{center}
\end{figure}

\medskip
{\bf Acknowledgment}. 
The author thanks the referee deeply for careful reading and invaluable suggestions. 

\section{Minimal unfolded region of a convex hull}
Let $\overline{xy}$ $(x,y\in\RR^\n)$ denote the line segment joining $x$ and $y$. 
Let $\widetilde\Om$ denote the convex hull of $\Om$ as before. 
We use the same notation as in the previous section. 

\begin{lemma}\label{prop}
Let $v\in S^{\n-1}$. 
We have $b\ge l_{\Om}(v)$ if and only if $\overline{xx'}\subset \Om$ for any $x\in \Om_{v,b}^+$, where $x'=\textrm{\rm R}_{v,b}(x)$. 
\end{lemma}
\begin{proof}
First observe that $\overline{xx'}=\{\textrm{\rm R}_{v,c}(x)\,:\,b\le c\le x\cdot v\}.$  
Therefore,  $\overline{xx'}\subset \Om$ for any $x\in \Om_{v,b}^+$ if and only if $R_{v,c}(x)\in\Om$ for any $c$ with $b\le c\le x\cdot v$ for any $x\in \Om_{v,b}^+$, which is equivalent to $R_{v,c}(x)\in\Om$ for any $x\in \Om_{v,c}^+$ for any $c\ge b$, which occurs 
if and only if $\textrm{R}_{v,c}(\Om_{v,c}^+)\subset\Om$ for every $c\ge b$, which is equivalent to $b\ge l_{\Om}(v)$. 
\end{proof}

\begin{lemma}\label{key_lemma}
Let 
\[
\widetilde\Om_{v,b}^+=\textrm{\rm cap}^+_{v,b}\left(\textrm{\rm conv}(\Om)\right)
=\textrm{\rm conv}(\Om)\cap\{x\in\RR^\n\,:\,x\cdot v> b\}
\]
for $v\in S^{\n-1}$ and $b\in\RR$. 
Then, for any $z\in\widetilde\Om_{v,{l_{\Om}(v)}}^+$ we have $\overline{zz'}\subset \widetilde\Om$, where $z'=\textrm{\rm R}_{v,{l_{\Om}(v)}}(z)$. 
\end{lemma}
\begin{proof}
Since Lemma \ref{prop} implies that $\overline{zz'}\subset\Om\subset \widetilde\Om$ for any $z\in\Om_{v,{l_{\Om}(v)}}^+$, we have only to show the assertion when $z\not\in\Om_{v,{l_{\Om}(v)}}^+$. 
Let $z\in\widetilde\Om_{v,{l_{\Om}(v)}}^+\setminus\Om_{v,{l_{\Om}(v)}}^+$. 
Then there are $x,y\in\Om$ such that $z\in\overline{xy}$. 
Put $x'=\textrm{\rm R}_{v,{l_{\Om}(v)}}(x)$ and $y'=\textrm{\rm R}_{v,{l_{\Om}(v)}}(y)$. 
Since $z\cdot v>l_{\Om}(v)$, there are exactly two cases: either both $x$ and $y$ are contained in $\Om_{v,{l_{\Om}(v)}}^+$ or one of them is contained in $\Om_{v,{l_{\Om}(v)}}^+$ and the other not. 

\smallskip
Case 1. Assume both $x$ and $y$ are contained in $\Om_{v,{l_{\Om}(v)}}^+$. 
Any point in $\overline{zz'}$ can be expressed as $z_t=tz+(1-t)z'$ for some $t$ $(0\le t\le 1)$. 
Put $x_t=tx+(1-t)x'$ and $y_t=ty+(1-t)y'$. 
Then $z_t\in\overline{x_ty_t}$ (Figure \ref{Uf_conv_hull-1}). 
Since Lemma \ref{prop} implies $x_t, y_t\in \Om$, we have $z_t\in\widetilde\Om$, which implies $\overline{zz'}\subset \widetilde\Om$. 

\begin{figure}[htbp]
\begin{center}
\begin{minipage}{.5\linewidth}
\begin{center}
\includegraphics[width=\linewidth]{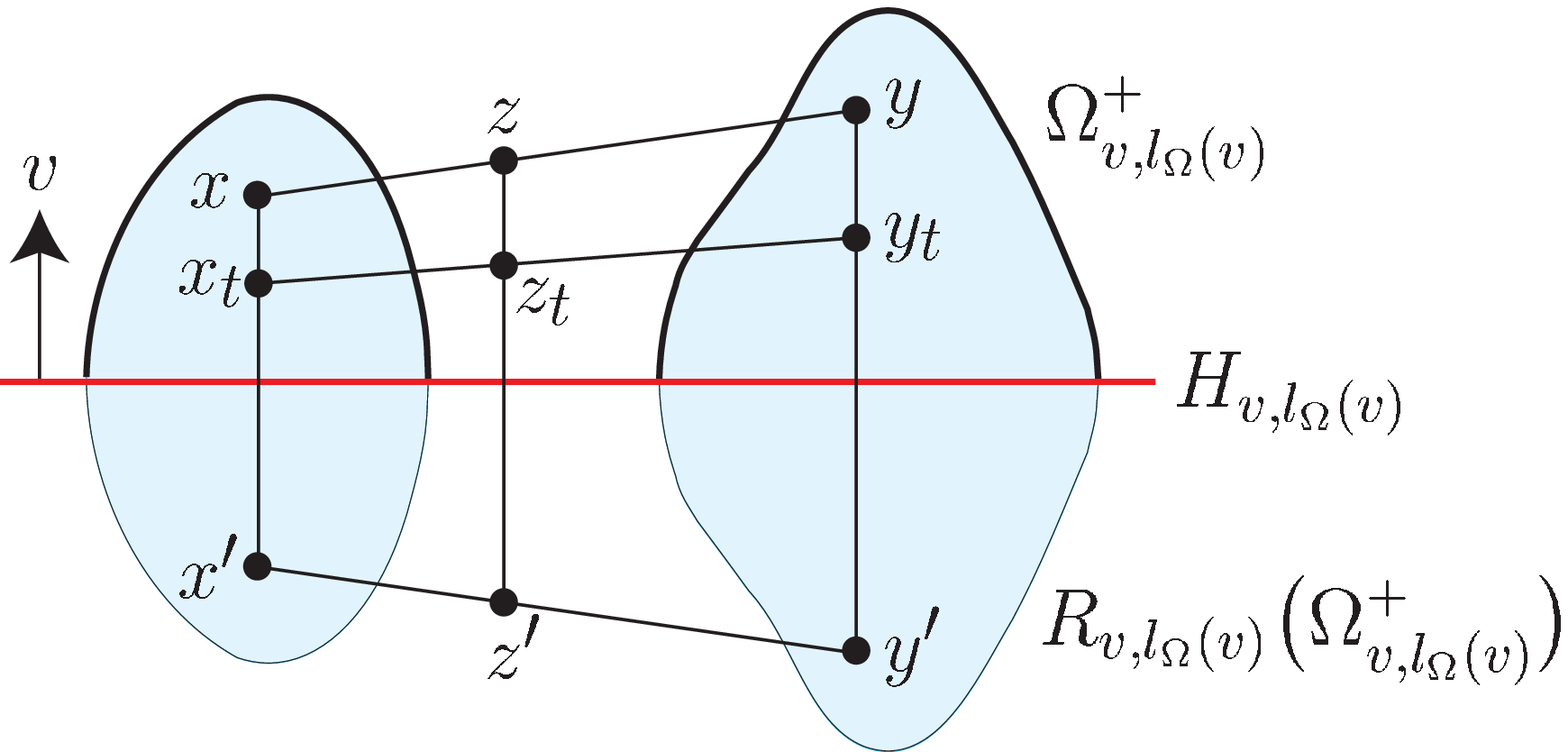}
\caption{}
\label{Uf_conv_hull-1}
\end{center}
\end{minipage}
\hskip 0.1cm
\begin{minipage}{.45\linewidth}
\begin{center}
\includegraphics[width=\linewidth]{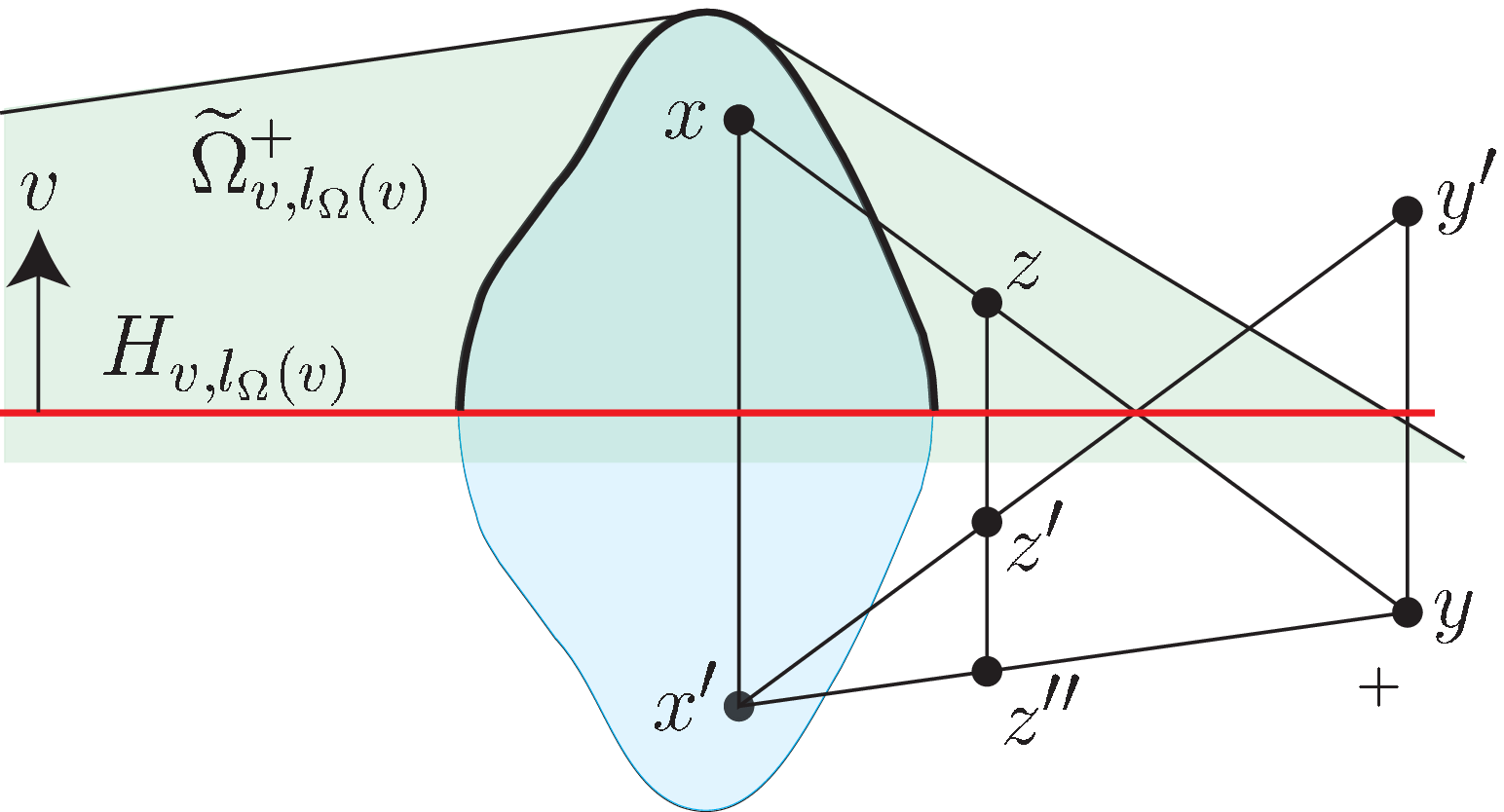}
\caption{}
\label{Uf_conv_hull-2}
\end{center}
\end{minipage}
\end{center}
\end{figure}

\smallskip
Case 2. Assume $x\in\Om_{v,{l_{\Om}(v)}}^+$ and $y\not\in\Om_{v,{l_{\Om}(v)}}^+$. 
Suppose $z$ can be expressed as $z=sx+(1-s)y$ for some $s$ $(0\le s\le 1)$. 
Put $z''=sx'+(1-s)y$. 
Then $\overline{zz'}\subset \overline{zz''}$ (Figure \ref{Uf_conv_hull-2}). 
We have $\overline{zz''}\subset \widetilde\Om$ by the same argument as in the above case, which  implies $\overline{zz'}\subset \widetilde\Om$. 
\end{proof}

\begin{remark}\rm 
Two operations, the convex hull operation $\Om\mapsto \textrm{\rm conv}(\Om)$ and the capping operation $\Om\mapsto \textrm{\rm cap}^+_{v,b}(\Om)$, do not commute in general. \\ \indent
Since $\textrm{\rm conv}\big(\textrm{\rm cap}^+_{v,b}(\Om)\big)\subset \{x\in\RR^\n\,:\,x\cdot v> b\}$ we have 
\[
\textrm{\rm conv}\big(\textrm{\rm cap}^+_{v,b}(\Om)\big) 
= \textrm{\rm cap}^+_{v,b}\left(\textrm{\rm conv}\big(\textrm{\rm cap}^+_{v,b}(\Om)\big)\right) 
\]
which implies 
\[
\textrm{\rm conv}\big(\textrm{\rm cap}^+_{v,b}(\Om)\big) \subset \textrm{\rm cap}^+_{v,b}\left(\textrm{\rm conv}(\Om)\right), 
\]
whereas the converse does not necessarily hold 
%
(see Figure \ref{conv-cap}).\\ 
\begin{figure}[htbp]
\begin{center}
\includegraphics[width=.7\linewidth]{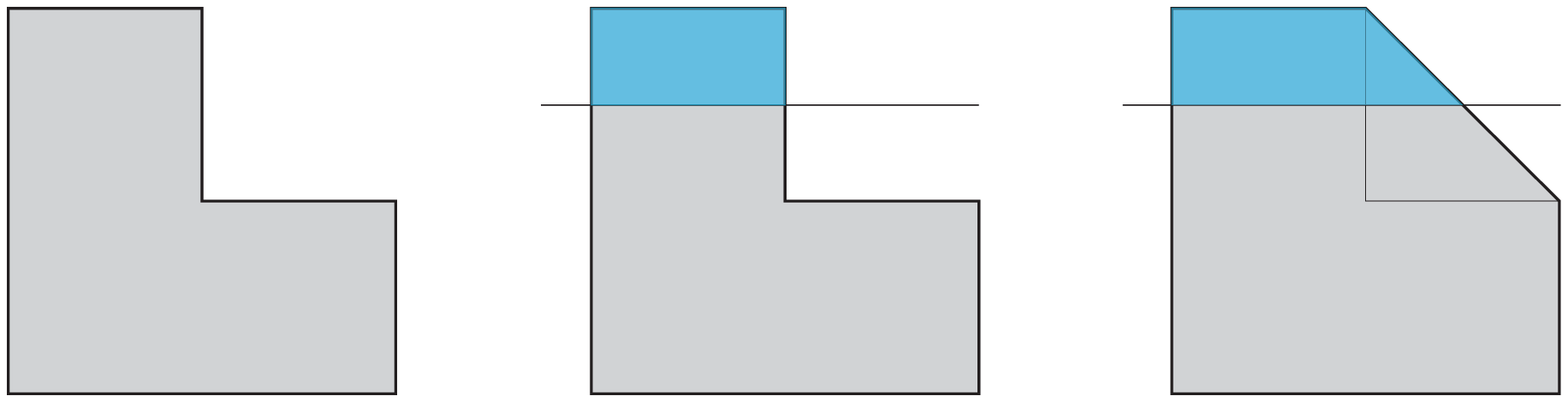}
\caption{$\Om$ (left), $\textrm{\rm conv}\big(\textrm{\rm cap}^+_{v,b}(\Om)\big)$ (center), and $\textrm{\rm cap}^+_{v,b}\left(\textrm{\rm conv}(\Om)\right)$ (right)}
\label{conv-cap}
\end{center}
\end{figure}
\indent
Of course, if $\Om$ is convex then $\textrm{\rm conv}\big(\textrm{\rm cap}^+_{v,b}(\Om)\big) = \textrm{\rm cap}^+_{v,b}\left(\textrm{\rm conv}(\Om)\right)$ for any $v$ and for any $b$. 
We remark that this property does not characterize the convexity, namely, even if $\textrm{\rm conv}\big(\textrm{\rm cap}^+_{v,b}(\Om)\big) = \textrm{\rm cap}^+_{v,b}\left(\textrm{\rm conv}(\Om)\right)$ for any $v$ and for any $b$, $\Om$ is not necessarily convex. 
This can be verified by considering an annulus $A=\{(\xi,\eta)\,:\,1\le \xi^2+\eta^2\le 2\}$. 
There is a counter example even if we restrict to the case when $\Om$ is homeomorphic to an $\n$-ball, for example, we can put $\Om={D}_2\times [0,1]\cup A\times[1,2]\subset\RR^3,$ where ${D}_2$ is a disc with center the origin and radius $2$. 
\end{remark}

\begin{theorem}\label{thm_convex_hull}
The minimal unfolded region of the convex hull of a compact subset $\Om$ is included in that of $\Om$. 
\end{theorem}
\begin{proof}
Lemma \ref{key_lemma} and Lemma \ref{prop} imply that $l_{\widetilde\Om}(v)\le l_{\Om}(v)$ for any $v\in S^{\n-1}$. 
Therefore 
\[
\textsl{Uf}\,(\widetilde\Om)
=\bigcap_{v\in S^{\n-1}}\left\{x\in\RR^\n\,:\,x\cdot v\le l_{\widetilde\Om}(v)\right\}
\subset\bigcap_{v\in S^{\n-1}}\left\{x\in\RR^\n\,:\,x\cdot v\le l_{\Om}(v)\right\}
=\textsl{Uf}\,(\Om).
\]
\end{proof}

\section{Minimal unfolded region of parallel bodies}
A $\d$-parallel body, or its generalization $\Om_\d$ to a non-convex subset given by \eqref{eq_def_parallel_body}, is a natural object not only in convex geometry but also in potential theory from a geometric viewpoint. The first half of the following is an improvement of Lemma 2.2 of \cite{O2}. 

\begin{theorem}\label{thm_parallel_body}
The minimal unfolded region of $\Om_\d$ is included in that of $\Om$ for every $\d>0$. 
Moreover, if $\Om$ is convex then the minimal unfolded region of $\Om_\d$ is same as that of $\Om$ for every $\d>0$. 
\end{theorem}

\begin{proof}
(1) To prove the first half, we have only to show $l_{\Om_\d}(v)\le l_\Om(v)$ for any $v\in S^{\n-1}$ for every $\d>0$. Fix $\d$ and $v$. Lemma \ref{prop} implies that it is enough to show that for any $x\in(\Om_\d)_{v,l_\Om(v)}^+$ there holds $\overline{xx'}\subset\Om_\d$, where $x'=R_{v,l_\Om(v)}(x)$. By the definition of $\d$-parallel body, there is $y\in\Om$ such that $x\in B_\d(y)$. 

\medskip
Case 1. Suppose $y\cdot v\le l_\Om(v)$. Then, since $\left(B_\d(y)\right)_{v,l_\Om(v)}^+$ is contained in the hemisphere of $B_\d(y)$, we have $R_{v,l_\Om(v)}\big(\left(B_\d(y)\right)_{v,l_\Om(v)}^+\big)\subset B_\d(y)$, which implies $\overline{xx'}\subset B_\d(y)\subset\Om_\d$ (Figure \ref{uf-parallel1}). 

\medskip
Case 2. Suppose $y\in\Om_{v,l_\Om(v)}^+$. Then, by Lemma \ref{prop} we have $\overline{yy'}\subset\Om$, where $y'=R_{v,l_\Om(v)}(y)$, which implies $\big(\,\overline{yy'}\,\big)_\d=\cup_{0\le t\le 1}B_\d(ty'+(1-t)y)\subset\Om_\d$. Since $x'\in B_\d(y')$ and $\big(\,\overline{yy'}\,\big)_\d$ is convex, we have $\overline{xx'}\subset \big(\,\overline{yy'}\,\big)_\d\subset\Om_\d$ (Figure \ref{uf-parallel2}).
\begin{figure}[htbp]
\begin{center}
\begin{minipage}{.49\linewidth}
\begin{center}
\includegraphics[width=\linewidth]{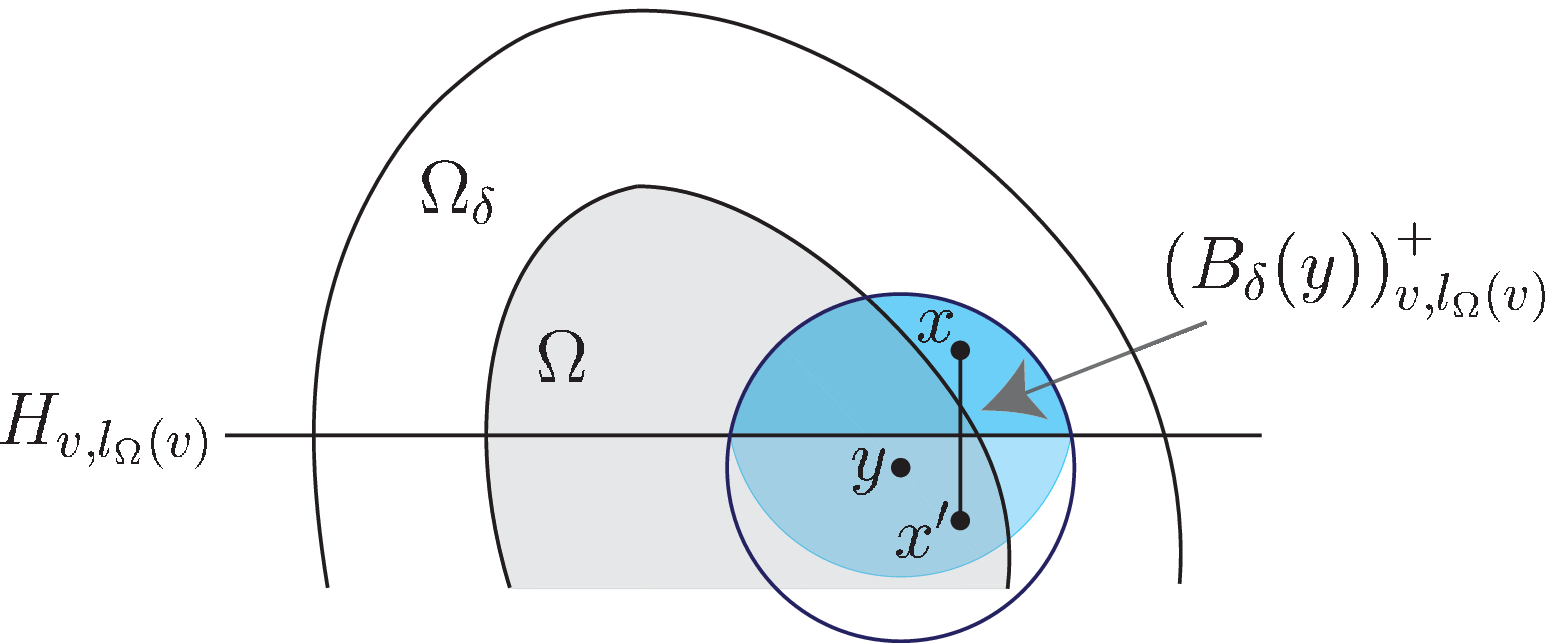}
\caption{}
\label{uf-parallel1}
\end{center}
\end{minipage}
\hskip 0.4cm
\begin{minipage}{.42\linewidth}
\begin{center}
\includegraphics[width=\linewidth]{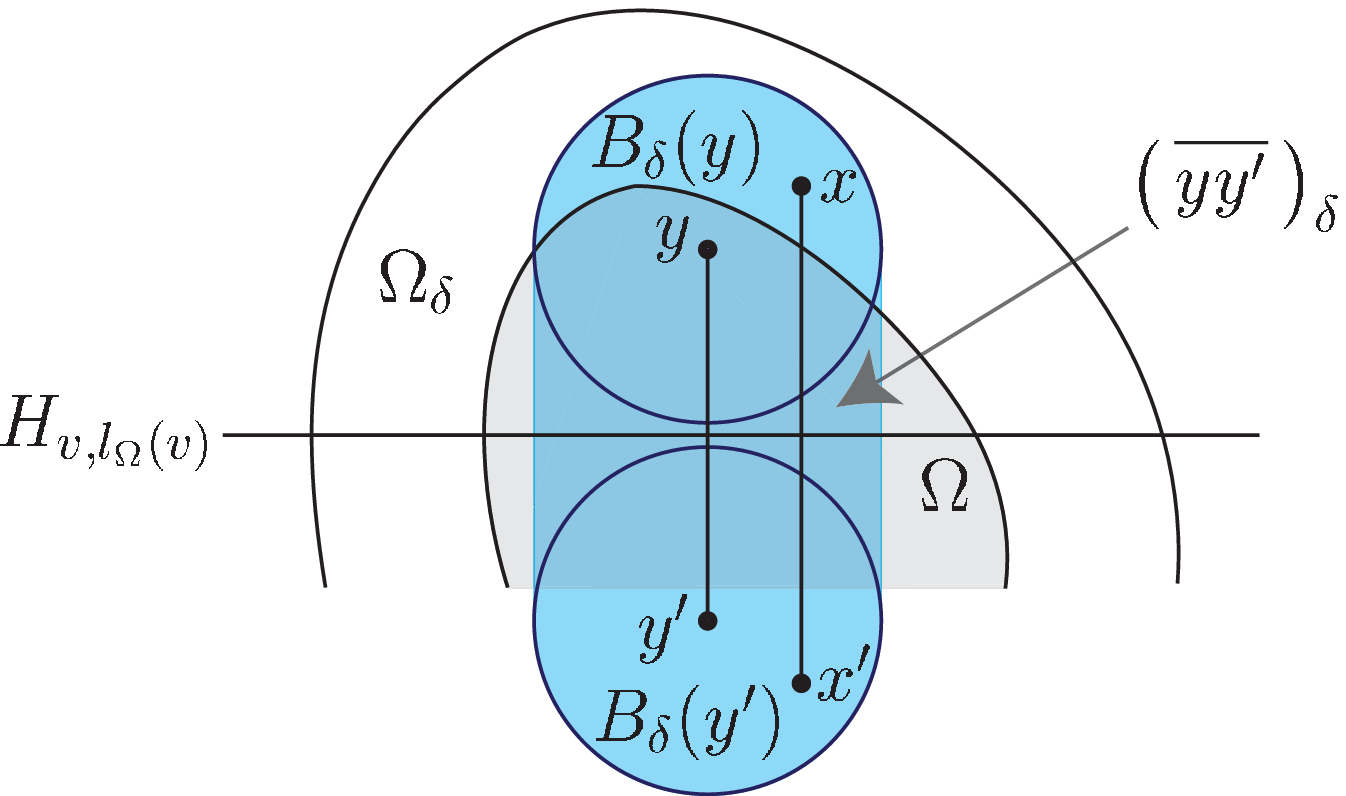}
\caption{}
\label{uf-parallel2}
\end{center}
\end{minipage}
\end{center}
\end{figure}

\smallskip 

(2) Suppose $\Om$ is convex. 
Let $\d>0$ and $v\in S^{\n-1}$. 
Lemma \ref{prop} implies that there is a point $x$ in $\Om_{v,l_\Om(v)}^+$ such that $\overline{xx'}\subset \Om$, where $x'=R_{v,l_\Om(v)}(x)$, and that $\overline{xx''}\not\subset \Om$ if $\overline{xx''}$ is obtained by extending $\overline{xx'}$ to the side of $x'$, i.e. $x''=R_{v,b}(x)$ for some $b<l_\Om(v)$. 
It follows that both $x$ and $x'$ are on the boundary of $\Om$. 
Let $\Pi'$ be a supporting hyperplane of $\Om$ at $x'$, and put $\Pi=R_{v,l_\Om(v)}(\Pi')$. 
In general, $\Pi'$ may not be uniquely determined, nevertheless $\Pi=R_{v,l_\Om(v)}(\Pi')$ is a supporting hyperplane of $\Om$ at $x$, as was pointed out in Theorem 4.12 of \cite{BMS}. This is because 
\[R_{v,l_\Om(v)}\left(\Om\cap\{x\in\RR^\n\,:\,x\cdot v< l_\Om(v)\}\right)\supset \Om\cap\{x\in\RR^\n\,:\,x\cdot v> l_\Om(v)\}.
\]
%
Let $n$ be the ``outer'' unit normal vector to $\Pi$ at $x$. 
Then we have $v\cdot n>0$ since $\overline{xx'}\subset\Om$ and $\overrightarrow{x'x}=cv$ for some positive number $c$. 
Put $y=x+\d n$ and $y'=R_{v,l_\Om(v)}(y)$. Then $B_\d(y)\cap\Om=\{x\}$ and $B_\d(y')\cap\Om=\{x'\}$. 
It implies that if we extend $\overline{yy'}$ to the side of $y'$ to obtain $\overline{yy''}$, i.e. $y''=R_{v,b}(y)$ for some $b<l_\Om(v)$, then $\overline{yy''}\not\subset\Om$, which implies $l_{\Om_\d}(v)\ge l_\Om(v)$ by Lemma \ref{prop}. 
Since we have $l_{\Om_\d}(v)\le l_\Om(v)$ as we saw in the proof of (1), we have $l_{\Om_\d}(v)= l_\Om(v)$. Since $v$ was arbitrary, we have $Uf(\Om_\d)=Uf(\Om)$. 
\end{proof}

\begin{figure}[htbp]
\begin{center}
\begin{minipage}{.45\linewidth}
\begin{center}
\includegraphics[width=\linewidth]{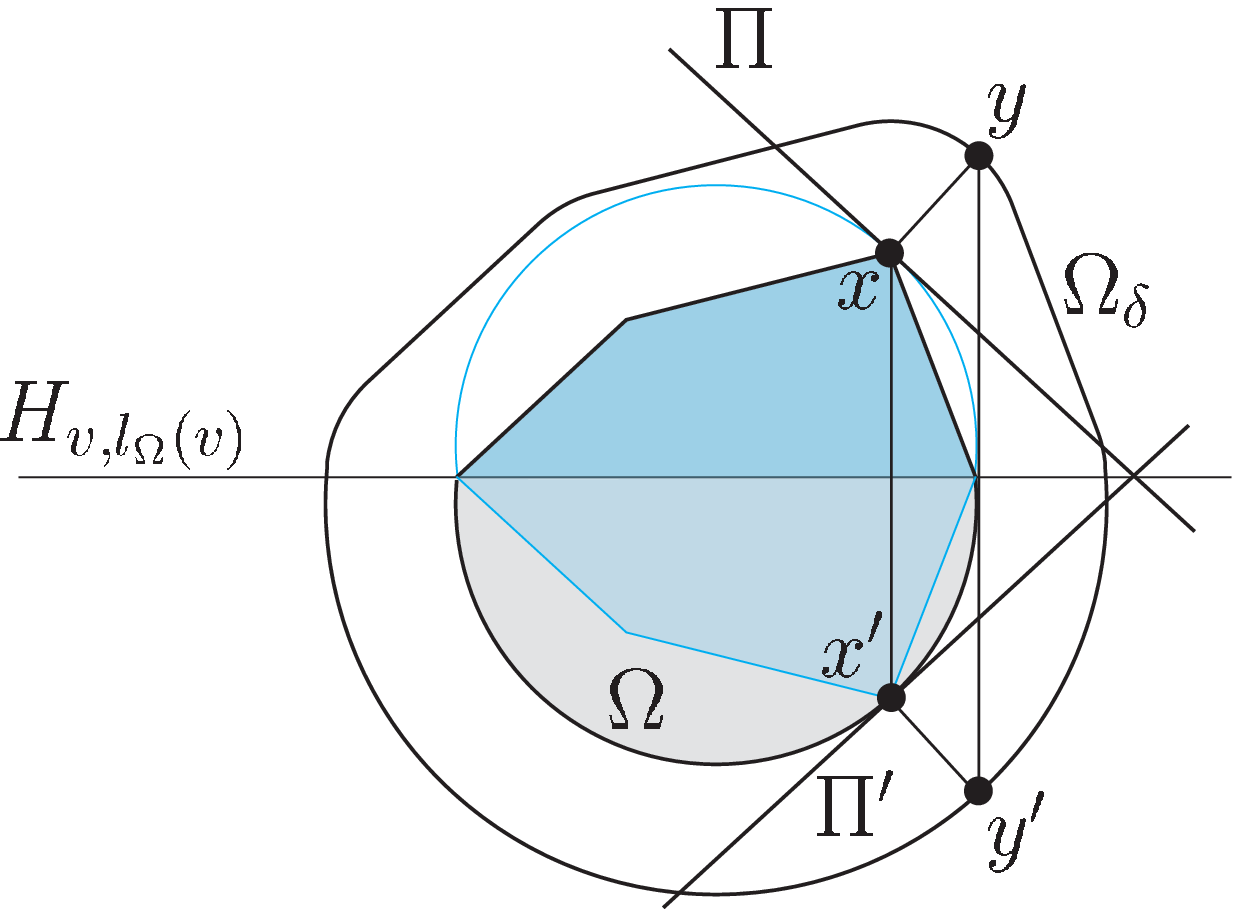}
\caption{}
\label{uf-parallel3}
\end{center}
\end{minipage}
\hskip 0.4cm
\begin{minipage}{.45\linewidth}
\begin{center}
\includegraphics[width=.64\linewidth]{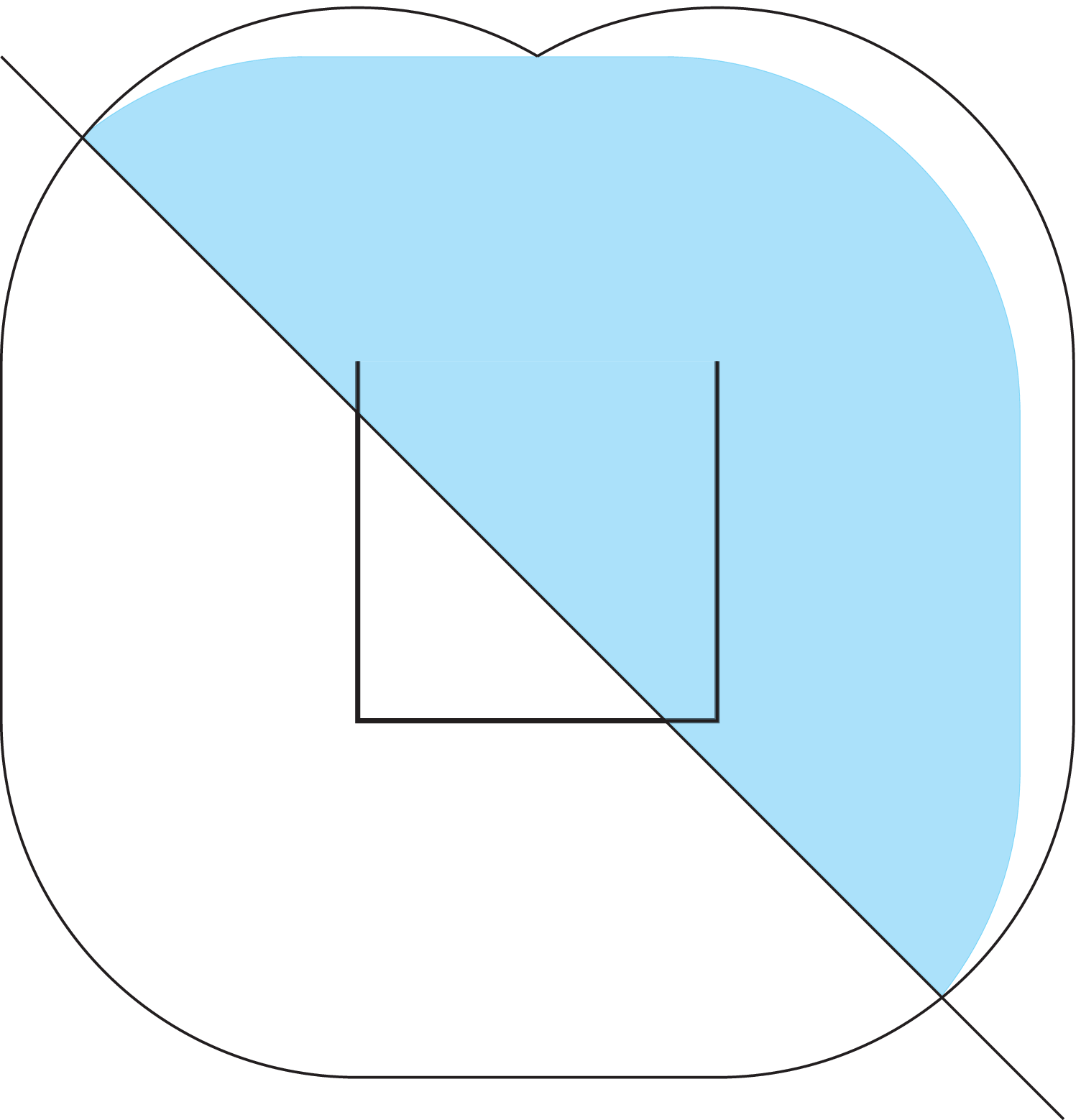}
\caption{$Uf(X_\d)\ne Uf(X)$}
\label{uf-parallel-eg}
\end{center}
\end{minipage}
\end{center}
\end{figure}

\begin{remark}\rm 
On the contrary, if $X$ is not convex, the parallel body could have a strictly smaller minimal unfolded region. 
Let $X=\{0\}\times[0,1]\cup[0,1]\times\{0\}\cup \{1\}\times[0,1]$ and $\d=1$. Then, $Uf(X)=[0,1]\times[0,1/2]$, whereas $Uf(X_\d)$ is much smaller, as is illustrated in Figure \ref{uf-parallel-eg}. 
\end{remark}

Department of Mathematics and Information Sciences, 

Tokyo Metropolitan University, 

1-1 Minami-Ohsawa, Hachiouji-Shi, Tokyo 192-0397, JAPAN. 

E-mail: ohara@tmu.ac.jp

\end{document}